\documentclass[11pt]{amsart}

\usepackage{amsthm,amssymb,graphicx,amsmath,pb-diagram,xcolor,colortbl,hyperref}
\usepackage{amsfonts}
\usepackage{verbatim}
\usepackage[T1]{fontenc}
\usepackage{graphicx}

\newtheorem{theorem}{Theorem}
\newtheorem{conjecture}[theorem]{Conjecture}
\newtheorem{corollary}[theorem]{Corollary}
\newtheorem{lemma}[theorem]{Lemma}
\newtheorem{proposition}[theorem]{Proposition}

\def\P{\mathbf {P}}
\def\E{\mathbf {E}}

\def\Z{\mathbb {Z}}

\begin{document}

\author{Robbert Fokkink, Ludolf Meester, and Christos Pelekis}


\title{Optimizing stakes in simultaneous bets}

\begin{abstract}
We want to find the convex combination $S$ of iid Bernoulli random
variables that maximizes $\P(S\geq t)$ for a given threshold~$t$. 
Endre Cs\'oka conjectured that such an $S$ is an average if $t\geq p$,
where $p$ is the success probability of the Bernoulli random variables. We prove this
conjecture for a range of $p$ and $t$.   
\end{abstract}

\maketitle

We study tail probabilities of convex combinations
of iid Bernoulli random variables. More specifically,
let $\beta_1, \beta_2, \ldots$ be an infinite sequence of 
independent Bernoulli random variables with success probability~$p$,
and let $t\geq p$ be a real number. We consider the
problem of maximizing $\P(\sum c_i\beta_i\geq t)$ over
all sequences $c_1, c_2, \ldots$ of non-negative reals
such that
$\sum c_i=1$. 
By the weak law of large numbers, 
the supremum of $\P(\sum c_i\beta_i\geq t)$ is equal
to $1$ if $t<p$. 
That is why we restrict our attention to $t\geq p$.

As a motivating example, consider a venture capitalist who
has a certain fortune $f$ to invest in any number of
startup companies. Each startup has an (independent) probability~$p$ of succeeding,
in which case it yields a return $r$ on investment.
If the capitalist 
divides his fortune into a (possibly infinite!) sequence $f_i$ of investments,
then his total return is $\sum rf_i \beta_i$. Suppose he wants to
maximize the probability that the total return reaches a threshold~$d$.
Then we get our problem with~$t=\frac d {rf}$. 

The problem has a how-to-gamble-if-you-must flavor~\cite{DS}: the capitalist places stakes $c_i$ on
a sequence of simultaneous bets. There is no need to place 
stakes higher than $t$. The way to go all out, i.e., \emph{bold play}, 
is to stake $t$ on $\lfloor \frac 1t\rfloor$
bets, but this is not a convex combination. That is why we say that
staking $\frac 1k$ on $k$ bets with $k={\lfloor \frac 1t\rfloor}$ is bold play.

In a convex combination $\sum c_i\beta_i$ we order $c_1\geq c_2\geq c_3\geq \ldots$.
We denote the sequence $(c_i)$ by $\gamma$ and write $S_\gamma=\sum c_i\beta_i$. 
We study the function
\begin{equation}\label{equ1}
\pi(p,t)=\sup \left\{\P\left(S_\gamma\geq t\right)\mid \gamma\right\}
\end{equation}
for $0\leq p\leq t\leq 1$.
It is non-decreasing in $p$ and non-increasing in $t$.
The following has been conjectured by Cs\'oka~\cite{C}, who was inspired by some
well known open problems in combinatorics:

\begin{conjecture}[Cs\'oka]\label{conj1}
For every $p$ and $t$ there exists a $k\in\mathbb N$ such that
$\pi(p,t)$ is realized by $c_i=\frac 1k$
if $i\leq k$ and $c_i=0$ if $i>k$ for some $k\in\mathbb N$.
In other words, the maximal probability is realized by an
average.
\end{conjecture} 

If the conjecture is true, then $\pi(p,t)$ is a binomial tail probability and we still need to determine
the optimal $k$. Numerical results of Cs\'oka suggest that bold play is optimal for most parameter values.

We are able to settle the conjecture for certain parameter values, as illustrated
in figure~\ref{fig:fig1} below.

\begin{figure}[htbp]
	\centering
		\includegraphics[width=0.50\textwidth]{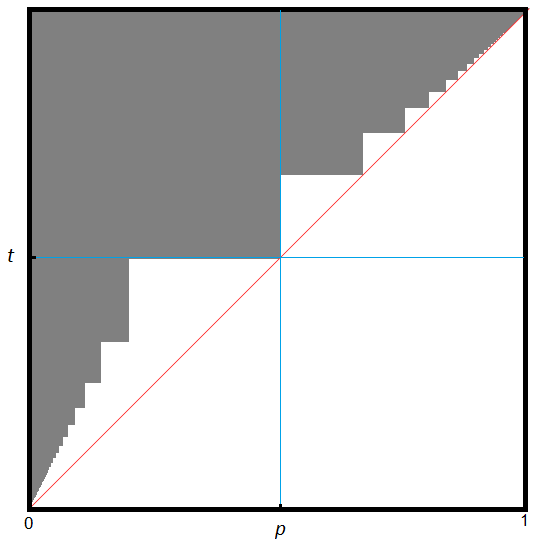}
	\caption{The shaded region represents all $(p,t)$ for which we are able
	to settle the conjecture. In all these cases bold play is optimal. Our results
	can be divided into three parts: favorable odds $p>\frac 12$, high threshold $t\geq \frac 12$,
	and unfavorable odds $p<\frac 12$.}
	\label{fig:fig1}
\end{figure}

It is natural to expect, though we are unable to prove this, that a gambler becomes
bolder if the threshold goes up or if the odds go down. In particular,
if $p'\leq p$ and $t'\geq t$ and if bold play is optimal for $(p,t)$, then it is natural to
expect
that bold play is optimal for $(p',t')$ as well. This is clearly visible in the figure
above, which is a union of rectangles with lower right vertices 
$(\frac k{k+1},\frac k{k+1})$ and $(\frac 1{2k+1},\frac 1{k+1})$ for $k\in\mathbb N$.
 
Our paper is organized as follows. We first lay the groundwork by analyzing properties
of $\pi(p,t)$ and prove that the supremum in equation~\ref{equ1} is a maximum. Then we
cover the shaded region in figure~\ref{fig:fig1} for the three separate parts of odds greater
than one, threshold greater than half, and odds smaller than one. Finally, we recall
an old result on binomial probabilities which would imply that (assuming Cs\'oka's conjecture
holds and bold play is stable in the sense that we just explained) bold play is optimal
if $p\leq \frac 1n\leq t$ for all $n\in\mathbb N$.  

\section{Related problems and results}

According to Cs\'oka's conjecture, if the coin is fixed and the stakes vary, then
the maximum tail probability is attained by a (scaled) binomial.
If the stake is fixed and the coins vary, then Chebyshev
already showed that the maximum probability is attained by a binomial:

\begin{theorem}[Chebyshev, \cite{Che}]\label{thmH}
For a given $s$ and $l$, let $Z=\beta_1+\cdots+\beta_l$ be any sum of $l$ independent Bernoullis
such that $\E[Z]=s$.
Then $\P(Z\geq t)$ is maximized by Bernoullis for which the success probabilities
assume at most three different values, only one of which is distinct from $0$ and $1$.
In particular, the maximum $\P(Z\geq t)$ is a binomial tail probability.
\end{theorem}

Samuels considered a more general situation with fixed expectations and arbitrary
random variables.

\begin{conjecture}[Samuels, \cite{Sam}]\label{conj2}
Let $0\leq c_1\leq \cdots\leq c_l$ be such that $\sum_{i=1}^l c_i<1$. 
Consider $\sup \mathbf P(X_1+\cdots+X_l\geq 1)$
over all collections of $l$ independent 
non-negative random variables such that $\mathbf E[X_i]=c_i$. This supremum
is a maximum which is attained by $X_j=c_j$ for $j\leq k$ and $X_j=(1-b)\beta_j$ for
$j>k$, where $k$ is an integer, the $\beta_j$ are Bernoulli random variables, and
$b=\sum_{i=1}^kc_i$. In other words, the gambler accumulates $b$ from small
expectations before switching to bold play.
 \end{conjecture} 

If all $c_i$ are equal, then the $\beta_j$ are identically distributed,
and the conjecture predicts that the maximum probability is attained by a binomial.
 
If one assumes that the Samuels conjecture holds, then one still needs
to determine the optimal~$k$. 
If $c_1=\ldots=c_l=\frac{1}{l+1}$ then the optimal~$k$ is equal to zero	~\cite{A}. 
This implies
that another well-known conjecture is a consequence of Samuels'
conjecture, see also~\cite{Paulin}.

\begin{conjecture}[Feige, \cite{Feige}]\label{conj3}
For all collections of $l$ independent 
non-negative random variables such that $\mathbf E[X_i]\leq 1$ 
it is true that
\[\mathbf P(X_1+\cdots+X_l<l+1)\geq \frac 1e.\]
\end{conjecture} 

As a step towards solving this conjecture, Feige proved the remarkable
theorem that there exists a $\delta>0$ such that
$\mathbf P(X_1+\cdots+X_l<l+1)\geq \delta$. 
His original value of $\delta=\frac 1{13}$
has been gradually improved. The current best result is $0.1798$ by Guo et al~\cite{Guo}.

\section{Properties of $\pi(p,t)$}

The function $\pi(p,t)$ is defined on a region bounded by a rectangular
triangle. It is easy to compute its value on the legs of the triangle:
$\pi(0,t)=0$ and $\pi(p,1)=p$. It is much harder to compute the
value on the hypothenuse. 

\begin{proposition}\label{pro1}
$\frac 12<\pi(p,p)<1$ if $0<p<1$.
\end{proposition}
\begin{proof}
We follow the proof of \cite[Lemma 1]{ABPY}.
The following Paley-Zygmund type inequality
for random variables of zero mean was proved in \cite[lemma 2.2]{HLNZ} and extended
in \cite{HZZ}:
\[
\P(X<0)\geq \left(2\sqrt 3 - 3\right)\frac{\E[X^2]^2}{\E[X^4]}.
\]
Applying this to $S_\gamma-p$ we have
\[
\P(S_\gamma<p)\geq \left(2\sqrt 3 - 3\right)\frac{\E[(S_\gamma-p)^2]}{\E[(S_\gamma-p)^4]}.
\]
The second moment of $S_\gamma-p$ is equal to $p(1-p)\sum c_i^2$ and the fourth
moment is equal to 
\begin{eqnarray*}
3p^2(1-p)^2\sum_{i\not=j} c_i^2c_j^2+(p(1-p)^4+p^4(1-p))\sum c_i^4
\end{eqnarray*}
This can be bounded by
\begin{eqnarray*}
\max\left(3, \frac 1{p(1-p)}-3\right)p^2(1-p)^2\left(\sum c_i^2\right)^2
\end{eqnarray*}
The Paley-Zygmund type inequality produces a lower bound on 
$\P(S_\gamma<p)$. Its complementary probability
$\pi(p,p)$ is bounded by:
\[
\pi(p,p)\leq 1-\frac {2\sqrt 3 - 3}{\max\left(3, \frac 1{p(1-p)}-3\right)}.
\]

It is possible to improve on this bound for small $p$ by using 
Feige's theorem. 
We write $S_\gamma=c_1\beta_1+S$. Then 
$\mathbf P(S_\gamma<p)\geq \mathbf P(\beta_1=0)\mathbf P(S<p)=(1-p)\mathbf P(S<p)$.
Note that $\mathbf E[S]=p(1-c_1)$ and we write $\mathbf P(S<p)=\mathbf P(S<\mathbf E[S]+pc_1)$.
We can approximate this probability by a truncated sum
$\mathbf P(S_n<\mathbf E[S]+pc_1)$, where $S_n=c_2\beta_2+\cdots+c_n\beta_n$ is a
sum of independent random variables of expectation $\leq pc_1$. By dividing
by $pc_1$ and applying the bound $0.1798$ of \cite{Guo} we find
\[
\pi(p,p)\leq 0.8202+0.1798p.
\]
 
We have two upper bounds. The first is more restrictive for large $p$ and the second is more
restrictive for small $p$.
The lower bound follows from bold play. Let $k\in\mathbb N$ be such that $\frac{1}{k+1}<p\leq \frac 1k$.
If $\bar S_k$ is the average of $k\geq 1$ Bernoullis, then 
$\P(\bar S_k\geq p)=\P(\bar S_k\geq \frac 1k)=1-(1-p)^k> 1-(1-\frac 1{k+1})^k$.
This is minimal and equal to $\frac 12$ if $k=1$.
\end{proof}

\begin{figure}[ht!]
	\centering
		\includegraphics[width=0.60\textwidth]{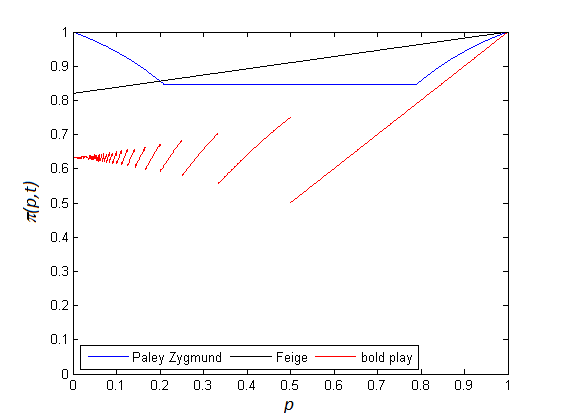}
	\caption{The upper bounds by Paley-Zygmund's inequality and Feige's theorem, and the lower bound by bold play, as in the proof of Proposition~\ref{pro1}.
	If Feige's conjecture holds, then the upper bound would
	meet the lower bound at $p=0$ in this figure.
		In corollary~\ref{cor1} we find $\pi(\frac 12,\frac 12)=\frac 34$, which is on the graph for the lower
	bound. If conjecture~\ref{conj1} is correct, then theorem~\ref{CBthm} implies that bold play is
	optimal for $p=t=\frac 1n$ for all $n$. These are the tops of the zigzag.}
	\label{fig:fig2}
\end{figure}

We say that a sequence $\gamma$ is \emph{finite} if $c_i=0$ for all but finitely
many~$i$, and infinite otherwise.

\begin{proposition}\label{pro2}
$\pi(p,t)=\sup \left\{\P\left(S_\gamma\geq t\right)\mid \gamma\ \mathrm{is\ finite}\right\}$
\end{proposition}
\begin{proof}
According to Jessen and Wintner's law of pure type~\cite[Theorem 3.5]{B}, 
either $\mathbf P(S_\gamma=s)=0$ for each $s\in\mathbb R$ or there exists
a countable set $\mathcal C$ such that $\mathbf P(S_\gamma\in\mathcal C)=1$.
In other words, the random variable $S_\gamma$ is either non-atomic or discrete. 
If $X$ and $Y$ are independent, and if $X$ is non-atomic, 
then the convolution formula implies that $X+Y$
is non-atomic.

Suppose that $\gamma$ is infinite. We prove that $S_\gamma$ is non-atomic.
Let $(c_{i_j})$ be a subsequence such that $c_{i_j}>2\sum_{k=j+1}^\infty c_{i_k}$.
Let $I$ be the set of all $i_j$ and let $J$ be its complement. 
Then both $S_I=\sum_I c_i\beta_i$ and $S_J=\sum_J c_i\beta_i$ 
are either discrete or non-atomic. 
By our choice of $c_{i_j}$, $S_I$ has the property that its value determines
the values of all $\beta_i$ for $i\in I$. This implies that $S_I$ is non-atomic.
Therefore, $S_\gamma=S_I+S_J$ is non-atomic. In particular
$\P(S_\gamma\geq t)=\P(S_\gamma>t)$.

Denote a truncated sum by $S_{\gamma,n}=\sum_{i\leq n}c_i\beta_i$. 
By monotonic convergence $\P\left(S_\gamma>t\right)=\lim_{n\to\infty}\P(S_{\gamma,n}>t)$.
Therefore, for any infinite $\gamma$, $\P(S_\gamma\geq t)$  
can be approximated 
by tail probabilities of finite convex combinations. 
\end{proof}

Cs\'oka conjectures that the tail probability is maximized by an average. 
This would imply that the sup in proposition~\ref{pro2} is a max. We are 
unable to prove this, but we can prove that the sup in equation~\ref{equ1} is a max.

\begin{theorem}\label{thm0}
If $p<t$ then $\pi(p,t)=\P(S_{\gamma}\geq t)$ for some $\gamma$.
Furthermore, $\pi(p,t)$ is left-continuous in $t$.
\end{theorem}
\begin{proof}
We write $\pi(p,t^-)=\lim_{s\uparrow t}\pi(p,s)$.
Since $\pi(p,t)$ is decreasing in $t$, it suffices to show that there
exists an $S_\gamma$ such that $\P(S_\gamma\geq t)\geq \pi(p,t^-)$.
Let $\gamma_n=(c_{n,i})_i$ be such that $\P(S_{\gamma_n}\geq t_n)$ converges to 
$\pi(p,t^-)$ for an increasing sequence $t_n\uparrow t$.  
By a standard diagonal argument we can assume that $(c_{n,i})_n$ is convergent
for all~$i$. Let $c_i=\lim_{n\to\infty} c_{n,i}$ and let $\gamma=(c_i)_i$.  
Then $\gamma$ is a non-increasing sequence which adds up to $\sum c_i=1-c\leq 1$.
Observe that $\gamma$ cannot be the all zero sequence, since this would imply that
$c_{n,1}\to 0$ and $\mathrm{Var}(S_{\gamma_n})=p(1-p)\sum c_{n,i}^2\leq p(1-p)c_{n,1}\to 0$,
so $S_{\gamma_n}$ converges to $p$ in distribution. Since we limit ourselves to $p<t$,
this means that $\P(S_{\gamma_n}\geq t)\to 0$ which is nonsense.
Therefore, $1-c>0$.

We first prove that $\pi(p,t^-)\leq \P(S_\gamma\geq t-cp)$.
Fix an arbitrary $\epsilon>0$. Let $i_0$ be such that $\sum_{j\geq i_0}c_{j}<\frac \epsilon 4$
and $c_{i_0}<\epsilon^4$.
Let $n_0$ be such that $\sum_{j\leq i_0}|c_{n,j}-c_j|<\frac\epsilon 4$ and $c_{n,i_0}<\epsilon^4$
for all $n\geq n_0$.
Now 
\[
\textstyle
\{S_{\gamma_n}\geq t_n\}\subset
\left\{\sum_{j\leq i_0}c_{n,j}\beta_j\geq t_n-cp-\epsilon \right\}
\bigcup
\left\{\sum_{j\geq i_0}c_{n,j}\beta_j\geq cp+\epsilon \right\}
\]
so by our assumptions
\[
\textstyle
\begin{array}{rcl}
\{S_{\gamma_n}\geq t_n\}&\subset&
\left\{\sum_{j\leq i_0}c_{j}\beta_j\geq t_n-cp-2\epsilon\right\}
\bigcup
\left\{\sum_{j\geq i_0}c_{n,j}\beta_j\geq cp+\epsilon\right\}\\
\\
&\subset&
\left\{S_\gamma\geq t_n-cp-2\epsilon\right\}
\bigcup
\left\{T_n\geq cp+\epsilon\right\}
\end{array}
\]
where we write $T_n=\sum_{j\geq i_0}c_{n,j}\beta_j$.
Observe that 
\[
\textstyle
\E[T_n]=\E[S_{\gamma_n}]-p\sum_{j<i_0}c_{n,j}<p-p\left(\sum_{j<i_0}c_j-\frac \epsilon 4\right)<p-p\left(1-c-\frac\epsilon 2\right)
<pc+\frac\epsilon 2\] 
and
\[\mathrm{Var}(T_n)=p(1-p)\sum_{j\geq i_0}c_{n,j}^2\leq c_{n,i_0}\sum_{j\geq i_0}c_{n,j}<\epsilon^4.\]
By Chebyshev's inequality, we conclude that
$\P\left(T_n\geq cp+\epsilon\right)<\epsilon$ for sufficiently small $\epsilon$.
It follows that
\[
\P(S_{\gamma_n}\geq t_n)\leq \P\left(S_\gamma\geq t_n-cp-2\epsilon\right)+\epsilon.
\] 
By taking limits $n\to\infty$ and $\epsilon\to 0$ we conclude that
\[
\pi(p,t^-)\leq\P\left(S_\gamma\geq t-cp\right).
\] 
Let $\bar\gamma=\frac 1{1-c}\gamma$. Then $S_{\bar\gamma}=\frac 1{1-c}S_\gamma$ is a convex combination
such that
\[\P\left(S_{\bar\gamma}\geq t\right)=
\P\left(S_\gamma\geq (1-c)t\right)\geq \P\left(S_\gamma\geq t-cp\right)\geq\pi(p,t^-).\]
Therefore, $\pi(p,t)=\P(S_{\bar\gamma}\geq t)$ and these inequalities are equalities. 
\end{proof}

We now more or less repeat this proof
to show that $\pi(p,t)$ is continuous in $p$.
Since we need to vary $p$,
we write $\beta^p$ for a Bernoulli with succes probability $p$
and $S_\gamma^p=\sum c_i\beta_i^p$.

\begin{theorem}\label{thmp}
$\pi(p,t)$ is continuous in $p$.
\end{theorem}
\begin{proof}
For any $\epsilon>0$ 
choose a finite $\gamma$ such that 
$\P(S_\gamma^p\geq t)\geq \pi(p,t) -\epsilon$.
If $p_n$ converges to $p$ then $\beta^{p_n}$ converges to
$\beta^p$ in probability. Since $\gamma$ is finite 
\[
\limsup_{n\to\infty}\pi(p_n,t)\geq
\lim_{n\to\infty} \P(S_\gamma^{p_n}\geq t)=\P(S_\gamma^p\geq t)\geq \pi(p,t)-\epsilon.
\]
It follows that
$\limsup_{n\to\infty} \pi(p_n,t)\geq \pi(p,t)$ for any sequence $p_n\to p$.
Since $\pi(p,t)$ is increasing in~$p$, it follows that $\pi(p,t)$ is left-continuous
in~$p$.

We need to prove right continuity, i.e., $\pi(p^+,t)=\pi(p,t)$.
This is trivially true on the hypothenuse, because this is the
right-hand boundary of the domain. Consider $p<t$. 
Let $p_n\downarrow p$ and $\gamma_n$ be such that
$\lim_{n\to\infty}\P(S_{\gamma_n}^{p_n}\geq t)=\pi(p^+,t)$.
By a standard diagonal argument we can assume that $\gamma_n$ converges
coordinatewise to some $\gamma$, which may not sum up to one. 
It cannot be the all zero sequence, i.e., the sum is not zero,
by the same argument as in the proof of theorem~\ref{thm0}.
The sequence $\gamma$ therefore sums up to $1-c$ for some $0\leq c<1$.
Again, we split $S^p_\gamma=H+T$ where $H=\sum_{j\leq i_0} c_j\beta_j^p$
and $T=\sum_{j>i_0} c_j\beta_j^p$. 
We choose $i_0$ such that $\mathbf E[T]<\frac\epsilon 4$ and 
$c_{i_0}<\epsilon^4$.
Similarly, $S_{\gamma_n}=H_n+T_n$
where $H_n$ converges to $H$ in probability, $\mathbf E[T_n]<pc+\frac\epsilon 2$
and	$\mathrm{Var}(T_n)<\epsilon^4$ for sufficiently large~$n$.
As in the previous proof, 
Chebyshev's inequality and convergence in probability imply
\[
\P(S_{\gamma_n}^{p_n}\geq t)\leq \P(H_n\geq t-cp-\epsilon)+\epsilon\leq \P(H\geq t-cp-2\epsilon)+\epsilon.
\]
for sufficiently large $n$.
By taking limits $n\to\infty$ and $\epsilon\to 0$ 
it follows that $\pi(p^+,t)\leq \P(S_\gamma^p\geq t-cp)$.
If we standardize $\gamma$ to a sequence $\bar\gamma$ so that we
get a convex combination, we again find that $\pi(p^+,t)\leq \P(S_{\bar\gamma}^p\geq t)$.
\end{proof}

\section{Favorable odds}

We consider $\frac 12\leq p < t$. In this case, bold play comes down to
a single stake $c_1=1$. 
We say that $I\subset \Z/n\Z$ is an \emph{interval of length $a<n$} if 
$I=[b,b+a)=\{b,b+1,\ldots,b+a-1\}$
for some $b$, which we call the initial element. 
We say that two intervals $I$ and $J$ are \textit{separate} if $I\cup J$ is
not an interval.  If $\mathcal F$ is a family of sets, then we write $\bigcup\mathcal F$
for the union of all these sets.

\begin{lemma}\label{pro3}
Let $\mathcal F$ be a family of $k$ intervals of length $a$ in $\Z/n\Z$
such that $\bigcup \mathcal F$ is a proper subset.
Then $|\bigcup{\mathcal F}|\geq k+a-1$.
\end{lemma}

\begin{proof}
$\bigcup \mathcal F$
is a union of say $c\geq 1$ separate intervals, all of lengths $\geq a$.
Any interval of length $b\geq a$ contains $b-(a-1)$ intervals of length $a$.
Therefore, $\bigcup\mathcal F$ contains $|\bigcup\mathcal F|-c(a-1)$ intervals
of length $a$.
It follows that $|\bigcup\mathcal F|-c(a-1)\geq k$.
\end{proof}

Two families $\mathcal F$ and $\mathcal G$ are \textit{cross-intersecting}
if $I\cap J\not=\emptyset$ for all $I\in\mathcal F$ and $J\in\mathcal G$.

\begin{lemma}\label{pro4}
Let $\mathcal F$ be a family of $k$ intervals of length $k$ in $\Z/n\Z$.
Let $\mathcal G$ be a family of intervals of length $a\leq n-k$ such
$\mathcal F$ and $\mathcal G$ are cross-intersecting.
Then $|\mathcal G|\leq a$.
\end{lemma}
\begin{proof}
Let $I=[b,b+k)$ be any element in $\mathcal F$.
An interval $[c,c+a)$ intersects $I$ if and only
if $c\in [b-a+1,b+k)$, which is an interval 
of length $k+a-1$. Therefore, the set $\mathcal I$ of initial elements $c$ of intervals
in $\mathcal G$ is contained in an intersection of $k$ intervals
of length $k+a-1$. The complement of $\mathcal I$ thus contains a union
of $k$ intervals of length $n-k-a+1$. 
By the previous lemma, this union has cardinality $\geq n-a$.
Therefore, $\mathcal I$ contains at most $a$ elements.
\end{proof}

\begin{lemma}\label{pro6}
Let $(V,\mu)$ be a finite measure space such that $\mu(V)=b$ and 
let $V_i\subset V$ for $i=1,\ldots, k$ be such that $\mu(V_i)\geq t$.
Then $\mu\left(\bigcap V_i\right)\geq kt-(k-1)b$.
\end{lemma}
\begin{proof}
\[
\mu\left(\bigcap V_i\right) = b - \mu\left(\bigcup V_i^c\right) \geq b 
- \sum (b -\mu(V_i)) \geq kt -(k-1)b.
\]
\end{proof}

\begin{theorem}\label{thm2}
If $\frac {k}{k+1}< p\leq \frac {k+1}{k+2}<t$
for some positive integer $k$, then bold play is optimal.
\end{theorem}

\begin{proof}
Bold play gives a probability $p$ of reaching $t$. We
need to prove that $\P(S_\gamma\geq t)\leq p$ for arbitrary $\gamma$.
By proposition~\ref{pro2} we may assume that $\gamma$ is finite. It
suffices to prove that $\P(S_\gamma \geq t)\leq p$ for rational $p$, since
$\pi(p,t)$ is monotonic in $p$.

Let $n$ be the number of non-zero $c_i$ in $\gamma$ and let~$p=\frac ab$. 
Let $X_i$ be a sequence of
$n$ independent discrete uniform $U\{0,b-1\}$ random variables, i.e,
$X_i=c$ for $c\in\{0,\ldots,b-1\}$ with probability $\frac 1b$. 
Let $B_i^0=1_{[0,a)}(X_i)$
for $1\leq i\leq n$.
Then $S_{\gamma}$ and $Y^0=\sum c_i B_i^0$ are identically distributed.
Think of $c_iB_i^0$ as an assignment of weight $c_i$ to a random element
in $\{0,\ldots,b-1\}=\mathbb Z/b\mathbb Z$. 
Let $\ell(j)$ be the sum of the coefficients -- the load --
that is assigned to~$j\in\mathbb Z/b\mathbb Z$. 
Then $Y^0=\ell(0)+\cdots+\ell(a-1)$, i.e., $Y^0$ is the load of $[0,a)$.
Instead of $[0,a)$ we might as well select any interval $[j,j+a)\subset\mathbb Z/b\mathbb Z$.
If $Y^j$ is the load of $[j,j+a)$, then $S_{\gamma}\sim Y^j$,
and $\P(S_\gamma\geq t)=\frac 1b \sum \P(Y^j\geq t)$.
We need to prove that $\sum \P(Y^j\geq t)\leq a$.

Let $\Omega$ be the sample space of the $X_i$. For $\omega\in\Omega$,
let $J(\omega)$ be the cardinality of
$\mathcal J(\omega)=\{j\colon Y^j(\omega)\geq t\}\subset \mathbb Z/b\mathbb Z$. 
In particular, $\P(S_\gamma\geq t)=\frac 1b \mathbf E[J]$.
It suffices to prove that $J\leq a$. 
Assume that $J(\omega)\geq a$ for some $\omega\in\Omega$.
Apply lemma~\ref{pro6} to the counting measure to find 
\[\left|\bigcap_{l=0}^k (\mathcal J(\omega)-la)\right|\geq (k+1)a-kb.\]
Note that $i\in \mathcal J(\omega)-j$ if and only if $[i+j,i+j+a)$ has load~$\geq t$.
Therefore, there are at least $(k+1)a-kb$ elements $i$ such that 
the intervals $[i,i+a), [i+a,i+2a),\ldots,[i+ka,i+(k+1)a)$ all have load $\geq t$.
The intersection of these $k+1$ intervals is equal to
\[
I_{i}=[i,i+(k+1)a-kb)\] 
It has load
$\geq (k+1)t-k$
by lemma~\ref{pro6}. 
Its complement $I_i^c$ has load $\leq k+1-(k+1)t<t$.
If $j\in \mathcal J(\omega)$ then
$[j,j+a)$ has load $\geq t$ and therefore it
intersects $I_{i}$ for all $i\in \bigcap_{l=0}^k (J(\omega)-la)$.
There are $\geq (k+1)a-kb$ 
such intervals $I_i$, and we denote this
family by $\mathcal F$.
Let $\mathcal G$ be the family of $[j,j+a)$ with $j\in J(\omega)$. 
Lemma~\ref{pro4} applies since the length of $I_i$
is $(k+1)a-kb$ and since $a\leq b-\left((k+1)a-kb\right)$.
We conclude
that $J(\omega)\leq a$.
\end{proof}

With some additional effort, we can push this result to the hypothenuse.

\begin{proposition}\label{thm2diag}
If $p=t=\frac {k+1}{k+2}$
for some positive integer $k$, then bold play is optimal if $k>1$,
and $c_1=c_2=c_3=\frac 13$ is optimal if $k=1$.
\end{proposition}

\begin{proof}
By proposition~\ref{pro2} it suffices to prove that 
$\P(S_\gamma\geq t)\leq p$ for finite~$\gamma$.
We adopt the notation of the previous theorem.
Let $n$ be the number of non-zero coefficients in $\gamma$,
and let $X_i$ be $n$ random selections of $\{0,\ldots,k+1\}$.
We assign the coefficients according to these selections and 
let $Y^j=1-\ell(j)$ be the load of the 
set $\{0,\ldots,k+1\}\setminus\{j\}$.
Each $Y^j$ is identically distributed to $S_\gamma$.
For $\omega\in\Omega$ let $J(\omega)$ be the number of $Y^j(\omega)$
that reach the threshold, or equivalently, the number
of loads $\ell(j)\leq \frac 1{k+2}$. 
We have $\frac 1{k+2}\mathbf E[J]=\P(S_\gamma\geq t)$.
In the proof above, we showed that $J\leq k+1$ if $t>p=\frac{k+1}{k+2}$. 
This is no longer true now that we have $t=p$.
It may happen that $J(\omega)=k+2$ in which case
all $Y^j(\omega)$ are equal to $\frac {k+1}{k+2}$
and all loads $\ell(j)$ are equal to $\frac 1{k+2}$.
Note that 
this can only happen if all $c_i$ are bounded by $\frac{1}{k+2}$, so $n\geq k+2$.

We think of the coefficients as being assigned one by one in increasing
order. In particular, $c_{n-1}$ and $c_n$ are placed last. 
If $J=k+2$, then either $k$ or $k+1$ of the 
loads are equal to $\frac 1{k+2}$ before $c_{n-1}$ and $c_n$ are placed. 
In the first case, there are two remaining loads $<\frac 1{k+2}$ and
the probability that $c_{n-1}$ are $c_n$ are placed here is $\frac 2{(k+2)^2}$.
In the second case, there is only one remaining load $<\frac 1{k+2}$ and
the probability that $c_{n-1}$ and $c_n$ are placed here is $\frac 1{(k+2)^2}$.
We conclude that 
$\P(J=k+2)\leq\frac{2}{(k+2)^{2}}$ and therefore
\[
\mathbf E[J]\leq (k+2)\P(J=k+2)+(k+1)\P(J<k+2)=k+1+\P(J=k+2)
\] 
is bounded by $k+1+\frac{2}{(k+2)^2}$.
Thus we obtain $\P(S_\gamma\geq t)\leq \frac {k+1}{k+2}+\frac{2}{(k+2)^{3}}$.
This bound is reached if $k=1$ and $c_1=c_2=c_3=\frac 13$.

Let $k>1$ and let $J(\omega)=k+2$. 
We first consider the case that $c_{n-1}\not=c_n$.
suppose there are two remaining loads $<\frac 1{k+2}$ before
$c_{n-1}$ and $c_{n-2}$ are placed, then each
$c_{n-1}$ and $c_n$ can only be assigned to a unique place
to complete all loads to $\frac 1{k+2}$.
Since $k>1$, there are at least
two loads $\ell(i_1)=\ell(i_2)=\frac 1{k+2}$ before the final
two coefficients are placed. Let $\bar\omega$ assign
$c_j$ for $j<n-1$ in the same way as $\omega$, but it
reassigns $c_{n-1}$ to $i_1$ and $c_{n}$ to $i_2$.
Then $J(\bar\omega)=k$, because the loads at $i_1$ and
$i_2$ exceed the threshold.
We can reconstruct $\omega$ from $\bar\omega$
because the loads at $i_1$ and $i_2$ are the only ones that exceed the
threshold for $\bar\omega$, and their values are different because $c_{n-1}\not=c_n$.
We have a $1-1$ correspondence between $\omega\in\{J=k+2\}$ and $\bar\omega\in\{J=k\}$.
Let $\mathcal E=\{J=k+2\}$ and let $\mathcal F=\{\bar\omega\colon \omega\in \mathcal E\}$.
Then $\P(\mathcal F)=\P(\mathcal E)$
and $\mathcal E\cap\mathcal F=\emptyset$. 
This implies that 
\[
\mathbf E[J]\leq(k+2)\P(\mathcal E)+k\P(\mathcal F)+(k+1)\P(\mathcal E^c\cap\mathcal F^c)\leq k+1.
\]
In particular $\P(S_\gamma\geq t)\leq \frac {k+1}{k+2}$ and bold play is optimal.

Finally, consider the remaining case $k>1$ and $J(\omega)=k+2$ and $c_{n-1}=c_n$. 
In this case, we may switch the assignments of $c_{n-1}$ and $c_n$ to complete the loads.  
Let $\omega'\in\Omega$ represent this switch (it may be equal to $\omega$ if the
assignments are the same). 
Again, let $i_1$ and $i_2$ be two locations for which the
loads have already been completed before $c_{n-1}$ and $c_n$ are placed.  
Let $\{\bar \omega,\bar{\omega}'\}$ be the elements which assign 
the first $n-2$ coefficients in the same way, but assigns the final
two elements to $i_1$ and $i_2$.
In particular, $J(\bar\omega)=J(\bar\omega')=k$.
We can reconstruct $\{\omega,\omega'\}$ from $\{\bar\omega,\bar\omega'\}$, the correspondence
is injective, so again $\P(\mathcal E)=\P(\mathcal F)$
and we conclude in the same way that bold play is optimal.   		
\end{proof}

If bold play is stable, as discussed above,
then proposition~\ref{thm2diag} would imply that bold play
is optimal if $p\leq\frac{k+1}{k+2}\leq t$ for $k>1$.
Establishing this stability of bold play appears to be hard, 
and we are able to settle only one specific case 
in the range $p<\frac{k+1}{k+2}=t$.

\begin{proposition}
If $p=\frac{2k+1}{2k+3}$ and $t=\frac {k+1}{k+2}$, then bold play is optimal if $k>1$.
\end{proposition}

\begin{proof}
We randomly distribute the coefficients of a finite $\gamma$ over $2k+3$ locations. 
Let $Y^j=\ell(j)+\ldots+\ell(2k+j)$ be the load of the discrete interval $[j,2k+j+1)$,
where as before we reduce modulo $2k+3$. Then $S_\gamma\sim Y^j$ and the sum of all
$Y^j$ is equal to $2k+1$. Let $J$ be the number of $Y^j$ that reach the threshold.
Not all $Y^j$ can reach the threshold and therefore $J\leq 2k+2$.
Then
\[\P(S_\gamma\geq t)=\frac {\mathbf E[J]}{2k+3}\leq \frac {2k+1}{2k+3}\P(J\leq 2k+1)+\frac{2k+2}{2k+3}\P(J=2k+2).
\] 
We need to prove that $\P(S_\gamma\geq t)\leq\frac{2k+1}{2k+3}$. If $\P(J=2k+2)=0$ then we are done.
Therefore, we may assume that $\P(J=2k+2)>0$. 
Only one of the $Y^j$ does not meet the threshold and without
loss of generality we may assume it is $Y^{2}$, which has load $1-\ell(0)-\ell(1)$. 
The other $Y^j$ reach the threshold, and since the sum of all $Y^j$
is equal to $2k+1$, we find that
\[
2k+1\geq (2k+2)t+1-\ell(0)-\ell(1).
\] 
In other words, $\ell(0)+\ell(1)\geq \frac 2{k+2}$. If $\ell(0)>\frac 1{k+2}$ then $Y^j$ does not
reach the threshold if it does not include $\ell(0)$. Only $2k+1$ of the $Y^j$ include $0$, contradicting
our assumption that $J=2k+2$. Therefore $\ell(0)\leq \frac 1{k+2}$ and since the same applies to $\ell(1)$
we have in fact that $\ell(0)=\ell(1)=\frac 1{k+2}$. Since $Y^j=1-\ell(j-1)-\ell(j-2)$ we have that 
all $\ell(i)+\ell(i+1)\leq \frac 1{k+2}$ other than $\ell(0)+\ell(1)$. In particular, $\ell(2)=\ell(2k+2)=0$.
We find that $Y^2=\frac k{k+2}$ and $Y^1=Y^{3}=\frac {k+1}{k+2}$. The sum of the remaining $Y^j$,
with $j\not\in\{1,2,3\}$ is at least $2kt$ and the sum of all $Y^j$ is $2k+1$. Since $2kt=2k+1-\frac k{k+2}-\frac{2k+2}{2k+1}$ all the remaining $Y^j$ have to be equal to $t$. 
It follows that the loads alternate between zero and $\frac 1{k+2}$: $\ell(i)=\frac 1{k+2}$ if $i$ is odd and
$\ell(i)=0$ if $i>0$ is even.
We conclude that if $J=k+2$ then all non-zero loads are equal and only two non-zero loads are consecutive.
There are exactly $2k+3$ such arrangements. There are also $2k+3$ arrangements in which the non-zero loads
are consecutive. In this case $J=k+2\leq 2k$. It follows that $\P(J\leq 2k)\geq P(J=2k+2)$, which
implies that $\E[J]\leq 2k+1$. Bold play is optimal.
\end{proof}

These results conclude our analysis of the upper right hand block of figure~\ref{fig:fig1}.
A zigzag of triangles along the hypothenuse remains. Numerical results of
Cs\'oka~\cite{C} suggest that bold play is optimal for all of these triangles,
except for the one touching on $\{(p,p)\colon \frac 12\leq p\leq\frac 23\}$.
In the next section we will confirm that bold play is not optimal for this triangle.

\section{High threshold}

We now consider the case $p\leq \frac 12<t$, when bold play comes
down to a single stake $c_1=1$.
We need to maximize $\mathbf P(S_\gamma\geq t)$ and we 
may assume that $\gamma$ is finite by proposition~\ref{pro2}.
Suppose that $\gamma$ has $\leq n$ non-zero coefficients,
i.e., $c_{n+1}=0$. 
Let $\mathcal F_{t,\gamma}$ be the family of $V\subset\{1,2,\ldots,n\}$, 
such that $\sum_{i\in V} c_i\geq t$.
Let $p(V)=p^{|V|}(1-p)^{n-|V|}$. Then
\begin{equation}\label{equ2}
\mathbf P(S_\gamma\geq t)=\sum_{V\in\mathcal F_{t,\gamma}}p(V).
\end{equation}
Therefore we need to determine the family $\mathcal F_{t,\gamma}$ that maximizes
the sum on the right hand side.
Problems of this type are studied in extremal combinatorics, see~\cite{Filmus}
for recent progress. 
A family $\mathcal F$ is \emph{intersecting} if no two elements are disjoint.
Two standard examples of intersecting families are $\mathcal F_1$, the
family of all $V$ such that $1\in V$, and $\mathcal F_{>n/2}$, the family
of all subsets such that $|V|>n/2$.  
Fishburn et al~\cite{Fish} settled the problem of
maximizing 
\[p(\mathcal F)=\sum_{V\in\mathcal F} p(V)\] 
over all
intersecting families $\mathcal F$:
\begin{theorem}[Fishburn et al]\label{thm3}
For a fixed $n$, let $\mathcal F$ be any intersecting family of subsets
from $\{1,\ldots,n\}$.
If $p\leq \frac 12$ then $p(\mathcal F)$ is maximized by $\mathcal F_1$.
If $p\geq \frac 12$ and $n$ is odd, then $p(\mathcal F)$ is maximized by
$\mathcal F_{>n/2}$.
\end{theorem}
\begin{proof}
Following~\cite{Fish}. 
First suppose $n$ is odd.
At most one of $V$ and $V^c$ can be in $\mathcal F$.
If $p\geq\frac 12$, then $p(V)>p(V^c)$ if $|V|>|V^c|$. 
Therefore $p(\mathcal F)$ is maximal if $\mathcal F$ contains
each set of largest cardinality.
It follows that $\mathcal F_{>n/2}$ maximizes $p(\mathcal F)$
if $n$ is odd and $p\geq \frac 12$. 
 
Now consider an arbitrary $n$ and $p\leq \frac 12$.
Let $c_a=|\mathcal F^a|$ be the cardinality of the subfamily 
$\mathcal F^a=\{V\in\mathcal F\colon |V|=a\}.$ 
At most one of $V$ and $V^c$ can be in $\mathcal F$ and 
therefore $c_a+c_{n-a}\leq \binom na$.
Since $p\leq\frac 12$ we have $p(V)>p(V^c)$ if $|V|<|V^c|$. 
For $a<\frac n2$ we want to maximize $c_a$ under the constraint
$c_a+c_{n-a}\leq \binom na$
(for $a=\frac n2$ it does not matter which of the two subsets we select,
as long as we select one of them).
By the Erd\H{o}s-Ko-Rado theorem, if $a<\frac n2$ then
$c_a$ is maximized by a family of subsets that contain one common element.  
For such a family, $c_a+c_{n-a}=\binom na$.
It follows that $p(\mathcal F_1)$ is maximal if $p\leq \frac 12$.
\end{proof}

Note that 
$\mathcal F_1$ corresponds to $\mathcal F_{t,\gamma}$ if $t>\frac 12$ and $c_1=1$, i.e., bold play.

\begin{corollary}\label{cor1}
If $p\leq \frac 12<t$ then bold play is optimal.
\end{corollary}
\begin{proof}
If $t>\frac 12$ then $\mathcal F_{t,\gamma}$ is intersecting. 
The maximizing family $\mathcal F_1$ corresponds to $\gamma$ with
$\gamma=(1,0,\ldots,0)$. This takes care of the upper left-hand block
in figure~\ref{fig:fig1}.
\end{proof}

The positive odds part of theorem~\ref{thm3} can be applied to the triangle
touching on $\{(p,p)\colon \frac 12\leq p\leq\frac 23\}$
that we mentioned above.  
The family $\mathcal F_{>n/2}$ can be represented by $\mathcal F_{t,\gamma}$
if $n=2k+1$ and $\frac 12<t\leq \frac {k+1}{2k+1}$, by taking
$\gamma=(\frac 1{2k+1},\ldots,\frac 1{2k+1})$.

\begin{corollary}\label{cor2}
If $\frac 12<p\leq t\leq \frac 23$ then bold play is not optimal.
\end{corollary}
\begin{proof}
Choose $k$ maximal such that $t\leq \frac{k+1}{2k+1}$ and set $n=2k+1$.
Then $\mathcal F_{>n/2}$ is the unique maximizer of $p(\mathcal F)$
and corresponds to 
$c_1=\cdots=c_{2k+1}=\frac 1{2k+1}$, which is not bold play.
\end{proof}

Finally, we settle the remaining part of the box in
the upper left corner of figure~\ref{fig:fig1}.

\begin{corollary}\label{cor3}
If $p\leq t=\frac 12$ then bold play is optimal.
\end{corollary}

\begin{proof}
Note that we may restrict our attention to $\gamma=(c_1,c_2,\ldots)$ such
that $c_1\leq \frac 12$ and that bold play corresponds to $(\frac 12,\frac 12,0,\ldots)$. 
\[
\begin{array}{ccl}
\P(S_\gamma\geq \frac 12)&=&p\P(S_\gamma\geq \frac 12\mid \beta_1=1)+(1-p)\P(S_\gamma\geq \frac 12\mid\beta_1=0)\\
&\leq&p+(1-p)\P(S_{\gamma}\geq \frac 12\mid\beta_1=0)
\end{array}
\]
If $\bar\gamma=\frac 1{1-c_1}(c_2,c_3,\ldots)$ then $\P(S_{\gamma}\geq \frac 12\mid\beta_1=0)=\P(S_{\bar\gamma}\geq \frac 1{2(1-c_1)})\leq p$ by theorem~\ref{thm3}.
We find that $\P(S_\gamma\geq\frac 12)\leq p+(1-p)p$ with equality for bold play.
\end{proof}

These results take care of the box in the upper left corner of figure~\ref{fig:fig1}, including
its boundaries.

\section{Unfavorable odds}

A family $\mathcal F\subset 2^{\{1,\ldots,n\}}$ has \emph{matching number} $k$, denoted by $\nu(\mathcal F)=k$, 
if the maximum number of
pairwise disjoint $V\in\mathcal F$ is equal to $k$. In particular, $\mathcal F$ is
intersecting if and only if~$\nu(\mathcal F)=1$. 
The matching number of $\mathcal F_{t,\gamma}$ is bounded by $\frac 1t$, because
$\sum_{j\in V}c_j\geq t$ for each $V\in\mathcal F_{t,\gamma}$ and $\gamma$ sums up to one.

A family $\mathcal F^u$ is \emph{$u$-uniform} if all its elements have cardinality $u$. 
According to the Erd\H{o}s matching conjecture~\cite{A, F, FK}, if $n\geq (k+1)u$
then the maximum cardinality $|\mathcal F^u|$
of a $u$-uniform family such that $\nu(\mathcal F^u)\leq k$ is either
attained by $\mathcal F_k^u$, the family of all $u$-subsets containing at least one element
from $\{1,\ldots,k\}$, or by $\mathcal F_{[(k+1)u-1]}^u$, the family containing
all $u$-subsets from $\{1,\ldots,(k+1)u-1\}$. Frankl~\cite{F} proved that $\mathcal F_k^u$
has maximum cardinality if $n\geq (2k+1)u-k$.
For recent progress on this conjecture,
see~\cite{FK} and the references therein.

\begin{theorem}\label{thm18}
If $p<\frac 1{2k+1}$ and $\frac 1{k+1}<t$ 
for some positive integer $k$
then bold play is optimal.
\end{theorem}
\begin{proof}
We need to prove that $\P(S_\gamma\geq t)\leq 1-(1-p)^k$ for finite~$\gamma=(c_1,c_2,\ldots)$.
For a large enough $n$ we have that $c_j=0$ if $j>n$. 
We have 
\[
\P(S_\gamma\geq t)=\sum_{\mathcal F_{t,\gamma}} p(V)=\sum_{j} |\mathcal F_{t,\gamma}^j|p^j(1-p)^{n-j}
\]
where $|\mathcal F_{t,\gamma}^j|$ denotes the number of subsets of cardinality $j$. 
By Frankl's result, we can put a bound on $|\mathcal F_{t,\gamma}^j|$ if $(2k+1)j-k\leq n$.
For larger $j$ we simply bound by $\binom{n}{j}$. In this way we get
that $\P(S_\gamma\geq t)$ is bounded by
\[
\sum_{j\leq {\frac {n+k}{2k+1}}} \left(\binom{n}{j}-\binom{n-k}{j}\right)p^j(1-p)^{n-j}
+
\sum_{j> {\frac {n+k}{2k+1}}}\binom{n}{j}p^j(1-p)^{n-j}
\]
which is equal to
\[
1-\sum_{j\leq {\frac {n+k}{2k+1}}}\binom{n-k}{j}p^j(1-p)^{n-j}
=1-(1-p)^k\P\left(X\leq \frac{n+k}{2k+1}\right)
\]
for $X\sim \mathrm{Bin}(n-k,p)$.
By our assumptions, there exists a $c<1$ such that $p<\frac c{2k+1}$.
If $n\to\infty$ then $\P\left(X\leq \frac{n+k}{2k+1}\right)\to 1$ since 
$\E[X]=(n-k)p<\frac{(n-k)c}{2k+1}$.
\end{proof}

We can push this result to the hypothenuse, using the same approach as
in the proof of proposition~\ref{thm2diag}, in one particular case: $p=t=\frac 13$.
By stability, one would expect that bold play is optimal for any $p<t=\frac 13$ but we can only
prove this for $p=\frac 1b$ for integers $b>3$.

\begin{proposition}\label{pro19}
Bold play is optimal if $t=\frac {1}{3}$ and $p=\frac 1b$ for an integer $b\geq 3$.
\end{proposition}
\begin{proof}
We may assume that $\gamma$ is finite and we randomly
assign its coefficients to $\{0,1,\ldots,b-1\}$. 
We denote $\ell(0)$, the load at zero, by $Y$ and
we need to prove that 
\[\P\left(Y\geq \frac 13\right)\leq p+(1-p)p+(1-p)^2p=r\]
which is the success probability of bold play if $t=\frac 13$. 
Let $K$ be the number of loads exceeding the threshold of $\frac 13$ before
the last two coefficients $c_{n-1}$ and $c_n$ are assigned.
Obviously, $K$ is either equal to $0$ or $1$ or $2$.
We will show that $\P(Y\geq \frac 13|K=j)\leq r$ for $j\in\{0,1,2\}$.

Suppose that $K=0$. In this case, if $Y$ reaches the threshold, then 
at least one of the two final coefficients has
to be placed in $0$. 
This happens with probability $p+(1-p)p<r$.

Suppose that $K=1$. One load has reached the threshold before the final
two coefficients are placed. This load is in $0$ 
with probability $p$. If the load is not in $0$, then
at least one of the remaining two coefficients has to be placed there.
This happens with probability $p+(1-p)p$. We conclude that
\[
\P\left(Y\geq \frac 13\,\middle\vert\, K=1\right)\leq p + (1-p)(p+(1-p)p)=r.
\]

Suppose that $K=2$. In other words, two loads have already reached the threshold before the final
two coefficients are placed. The probability that one of these two loads
is in $0$ is $2p$.
If none of the two loads is in $0$, then $\ell(0)$ can only reach the threshold if 
both remaining coefficients are assigned to $0$. The probability that this happens is $p^2$. 
\[
\P\left(Y\geq \frac 13\,\middle\vert\, K=2\right)\leq 2p+(1-2p)p^2\leq r
\] 
if $p\leq \frac 13$.
\end{proof}

\section{Binomial tails}

If conjecture~\ref{conj1} holds, then the tail probability is maximized by
a Bernoulli average $\bar X_k$ and we need to determine the optimal $k$.
It is more convenient to state this in terms of binomials. For a fixed $p$
and $t$, maximize 
\[
\mathbf P\left(\mathrm{Bin}(k,p)\geq kt\right)
\]
for a positive integer~$k$.  
Since the probability increases if $k$ increases and
$kt$ does not pass an integer,
we may restrict our attention to $k$ such that $kt\leq n< (k+1)t$
for some integer $n$. In other words, we need to only consider
$k=\lfloor \frac n t\rfloor$ for $n\in\mathbb N$.
If $t=\frac 1a$ is the reciprocal of an integer $a$, then the $k$ are multiples of $a$.
This is a classical problem. In 1693 John Smith asked which $k$ is optimal
if $a=6$ and $p=\frac 16$. Or in his original words, which of the following
events is most likely: fling at least one six with 6 dice, or at least two sixes with 12 dice,
or at least three sixes with 18 dice. The problem was communicated by Samuel Pepys
to Isaac Newton, who computed the probabilities. 
Chaundy and Bullard~\cite{CB} gave a very nice historical description 
(more history can be found in~\cite{KS,KS2}) and solved the problem,
see also~\cite{JS}. 

\begin{theorem}[Chaundy and Bullard]\label{CBthm}
For an integer $a>1$,
$\P(\mathrm{Bin}(ka,\frac 1a)\geq k)$ is maximal for $k=1$.
Even more so, the tail probabilities strictly decrease with $k$.
\end{theorem} 

In other words, if $p=t=\frac 1a$ and if Cs\'oka's conjecture holds,
then bold play is optimal.
By stability, one would expect that bold play is optimal
for $p\leq \frac 1a\leq t$. 
It turns out that it 
is possible to extend Chaundy and Bullard's theorem in this
direction and
prove that
$\P(\mathrm{Bin}(ka,p)\geq k)$ decreases with $k$
for arbitrary $p\leq \frac 1a$,
see~\cite[Theorem 1.5.4]{Pel}.

\section{Conclusion}

We settled Cs\'oka's conjecture for a range of parameters building on 
combinatorial methods. 
Cs\'oka's conjecture predicts that $\P(S_\gamma\geq t)$ attains
its maximum at an extreme point of the subset of positive
non-increasing sequences $\gamma$ in the unit ball in $\ell_1$.
Perhaps variational methods need to be considered.

Christos Pelekis was supported by the Czech Science Foundation
(GA\v{C}R project 18-01472Y), by
the Czech Academy of Sciences (RVO: 67985840),
and by a visitor grant of the Dutch mathematics cluster Diamant.

\bigskip
\noindent
Institute of Applied Mathematics\\
Delft University of Technology\\
Mourikbroekmanweg 6 \\
2628 XE Delft, The Netherlands\\
\texttt{r.j.fokkink@tudelft.nl, l.e.meester@tudelft.nl}
\\[4mm]
School of Electrical and Computer Engineering\\
National Technical University of Athens\\
Zografou, 15780, Greece\\ 
\texttt{pelekis.chr@gmail.com}

\begin{thebibliography}{AC}


\bibitem{A} 
N. Alon, P. Frankl, H. Huang, V. R\"odl, A. Ruci\'nski, B. Sudakov,
\newblock{Large matchings in 
uniform hypergraphs and the conjectures of Erd\H{o}s and Samuels.} 
\newblock{ J.  Combin. Th., Ser. A} 119 (2012) 1200--1215. 

\bibitem{ABPY}
I. Arieli, Y. Babichenko, R. Peretz, H. Peyton Young,
\newblock{
The speed of innovation in social networks.
}
\newblock {Econometrica} 88 no. 2 (2020), 569--594. 


\bibitem{B}
L. Breiman,
\newblock{Probability}, Reading, Addison-Wesley, 1968.

\bibitem{CB}
T.W. Chaundy, J.E. Bullard,
\newblock {John Smith's problem.}
\newblock {Math. Gazette} 44(350), (1960), 253--260.	

\bibitem{C}
E. Cs\'{o}ka,
\newblock{Limit theory of discrete mathematics problems.}
\newblock{arXiv1505.06984}.

\bibitem{DS} L.~E.~Dubins and L.~J.~Savage, 
\newblock{Inequalities for Stochastic Processes (How to Gamble If You Must).} 
\newblock{Dover}, 1976.

\bibitem{Feige} U.~Feige,
\newblock{On sums of independent random variables with unbounded variances, and estimating the average
degree in a graph.}
\newblock{SIAM J. Comput.} 35 (2006), 964--984.

\bibitem{Filmus}
Y. Filmus,
\newblock{The weighted complete intersection theorem.}
\newblock{J. Combin. Th. Ser. A} 151 (2017), 84--101.

\bibitem{Fish} 
P.C. Fishburn, P. Frankl, D. Freed, J. Lagarias, A.M. Odlyzko. 
\newblock {Probabilities 
for intersecting systems and random subsets of finite sets.}
\newblock{
SIAM. J. Algebraic Discrete Methods} 7 no. 1 (1986) 73--79.

\bibitem{F} 
P. Frankl, \newblock{Improved bounds on Erd\H{o}s' matching conjecture.}
\newblock{J. Combin. Th., ser. A}, 120 (2013) 1068--1072. 


\bibitem{FK}
P. Frankl, A. Kupavskii, \newblock{Some results around the Erd\H{o}s matching conjecture.} 
\newblock{Acta Math. Univ. Comenianae} 88 no. 3 (2019), 695--699.

\bibitem{Guo}
J. Guo, S. He, Z. Ling, Y. Liu,
\newblock{Bounding probability of small deviation on sum of independent random variables:
Combination of moment approach and Berry-Esseen theorem.} 
\newblock{arXiv:2003.03197}.

\bibitem{HLNZ}
S. He, Z. Q. Luo, J. Nie, S. Zhang,
\newblock{Semidefinite relaxation bounds for indefinite homogeneous quadratic optimization.}
\newblock {SIAM J. Optim.} 19 (2007) 503--523.

\bibitem{HZZ}
S. He, J. Zhang, S. Zhang,
\newblock{Bounding probability of small deviation: a fourth moment approach.}
\newblock{Math. Oper. Res.} 45 no. 1 (2010), 208--232.


\bibitem{JS}
K. Jogdeo, S.M. Samuels,
\newblock{Monotone Convergence of Binomial Probabilities and a Generalization of Ramanujan's
Equation.}
\newblock{Ann. Math. Statist.} 39 no. 4, (1968), 1191--1195.

\bibitem{KS}
T.H. Koornwinder, M.J. Schlosser,
\newblock {On an identity by Chaundy and Bullard, I.}
\newblock {Indag. Math. (NS)} 19, (2008), 239--261.	

\bibitem{KS2}
T.H. Koornwinder, M.J. Schlosser,
\newblock {On an identity by Chaundy and Bullard, II.}
\newblock {Indag. Math. (NS)} 24, (2013), 174--180.	


\bibitem{Paulin}
R.~Paulin,
\newblock{On some conjectures by Feige and Samuels.}
\newblock{arXiv1703.05152}.

\bibitem{Pel}
C. Pelekis,
\newblock{Search games on hypergraphs.}
\newblock{PhD thesis, TU Delft, 2014}. 

\bibitem{Sam}
S.M. Samuels,
\newblock{On a Chebyshev-type inequality for sums of independent random variables.}
\newblock{Ann. Math. Statist.} 37 no. 1, (1966), 248--259.



\bibitem{Che}
P. Tchebichef,
\newblock{D\'emonstration \'el\'ementaire d'une proposition g\'en\'erale de
la theorie des probabilit\'es.}
J. Reine. Angew. Math. 33 (1848), 259--267.	 

\end{thebibliography}
\end{document}